\documentclass[12pt,reqno,centertags,draft]{amsart}
\usepackage{amsmath}
\usepackage{}
\usepackage{color}
\usepackage{amsfonts}
\usepackage{a4wide}
\usepackage{amsmath,amsthm,amssymb,amscd}
\usepackage{latexsym}
\usepackage{hyperref}
\usepackage[numbers,sort&compress]{natbib}
\usepackage{hypernat}
\allowdisplaybreaks
\numberwithin{equation}{section}

\newtheorem{theorem}{Theorem}[section]
\newtheorem{proposition}[theorem]{Proposition}
\newtheorem{corollary}[theorem]{Corollary}
\newtheorem{lemma}[theorem]{Lemma}

\theoremstyle{definition}

\begin{document}
\title
[Positive  solutions to multi-critical Schr\"{o}dinger equations]
{Positive  solutions to multi-critical Schr\"{o}dinger equations}

\maketitle
\begin{center}
\author{Ziyi Xu, \ \  Jianfu Yang}
\footnote{ Email addresses: xuziyi@jxnu.edu.cn, jfyang200749@sina.com}
\end{center}
\begin{center}

\address{School of Mathematics
and Statistics, Jiangxi Normal University, Nanchang,
Jiangxi 330022, P. R. China }
\end{center}

\begin{abstract}
In this paper, we investigate the existence of multiple positive solutions to the following multi-critical Schr\"{o}dinger equation
\begin{equation}
\label{p}
\begin{cases}
 -\Delta u+\lambda V(x)u=\mu |u|^{p-2}u+\sum\limits_{i=1}^{k}(|x|^{-(N-\alpha_i)}* |u|^{2^*_i})|u|^{2^*_i-2}u\quad \text{in}\  \mathbb{R}^N,\\
\qquad\qquad\qquad u\,\in H^1(\mathbb{R}^N),
\end{cases}
\end{equation}
where $\lambda,\mu\in \mathbb{R}^+, \, N\geqslant 4$, and  $2^*_i=\frac{N+\alpha_i}{N-2}$ with $N-4<\alpha_i<N,\,i=1,2,\cdots,k$ are critical exponents and $2<p<2^*_{min}=\min\{2^*_i:i=1,2,\cdots,k\}$. Suppose that $\Omega=int\,V^{-1}(0)\subset\mathbb{R}^N$ is a bounded domain, we show that  for $\lambda$ large, problem \eqref{p} possesses at least $cat_\Omega(\Omega)$ positive solutions.

 {\bf Key words }:  Multi-critical Schr\"{o}dinger equation, \, multiple solutions, \, Lusternik-Schnirelman theory.

{\bf MSC2020:} 35J10, 35J20, 35J61

\end{abstract}

\bigskip

\section{Introduction}

\bigskip

In this paper, we investigate the existence of multiple positive solutions to the following multi-critical Schr\"{o}dinger equation
\begin{equation}
\label{p0}
\begin{cases}
 -\Delta u+\lambda V(x)u=\mu |u|^{p-2}u+\sum\limits_{i=1}^{k}(|x|^{-(N-\alpha_i)}* |u|^{2^*_i})|u|^{2^*_i-2}u\quad \text{in}\  \mathbb{R}^N,\\
\qquad\qquad\qquad u\,\in H^1(\mathbb{R}^N),
\end{cases}
\end{equation}
where $\lambda,\mu\in \mathbb{R}^+,\, N\geqslant 4$,  and $2^*_i=\frac{N+\alpha_i}{N-2}$ with $N-4<\alpha_i<N, \, i=1,2,\cdots,k$ are critical Hardy-Littlewood-Sobolev exponents and $2<p<2^*_{min}=\min\{2^*_i:i=1,2,\cdots,k\}$.

In the case $k=1$, problem \eqref{p0} is so-called Choquared equation. It has been extensively studied for subcritical and critical cases in \cite{Gao, Goel, MZ, Marco, Moroz, Moroz2} and references therein. Particularly, multiple solutions for the nonlinear Choquard problem
\begin{equation}\label{eq:1.1}
\left\{\begin{aligned}
-\Delta u +\lambda u& =(|x|^{-(N-\alpha)}*|u|^q)|u|^{q-2}u\quad {\rm in}\quad \Omega, \\
&u\in H^1_0(\Omega)\\
\end{aligned}\right.
\end{equation}
in the bounded $\Omega$ are obtained in \cite{Marco} by Lusternik-Schnirelman theory. It was proved that if $q$ is closed to the critical exponent $2^*_\alpha = \frac{N+\alpha}{N-2}$, problem \eqref{eq:1.1} possesses at least $cat_\Omega(\Omega)$ positive solutions; similar results were obtained in \cite{Goel}  for the critical case $q= 2^*_\alpha$. This sort of results
extend early works for elliptic problem with local nonlinearities by \cite{Alves,Bahri,Benci1,Benci2,Coron, Warner} etc.

A counterpart  for the Schr\"odinger equation
\begin{equation}\label{eq:1.2}
-\epsilon^2\Delta u+V(x)u=|u|^{p-2}u\quad \text{in}\ \mathbb{R}^N
\end{equation}
were considered in \cite{Chabrowski} and  \cite{Cingolani}. Although the topology of whole space $\mathbb{R}^N$ seems no effect of the existence of solutions,  a graph of the potential function $V(x)$ actually affects the existence of the number of solutions. Precisely, assume, among other things, that $\Omega:=V^{-1}(0)=\{x\in\mathbb{R}^N:V(x)=0\}$ is bounded, it is proved in \cite{Cingolani} for the subcritical case and in
\cite{Chabrowski} for the critical case that problem \eqref{eq:1.2} possesses at least $cat_\Omega(\Omega)$ positive solutions if $\varepsilon>0$ is small.

Analogous problem for the Schr\"odinger equation
\begin{equation}\label{w2}
  -\Delta u+(\lambda V(x)+1)u=u^{p-1}\quad u\in H^1(\mathbb{R}^N),
\end{equation}
was investigated in \cite{Wang} and \cite{Ding} for subcritical and critical cases respectively. Under certain conditions on $V(x)$, it is obtained in  \cite{Wang} and \cite{Ding} $cat_\Omega(\Omega)$ positive solutions for problem \eqref{w2} if $\lambda>0$ large. Various extensions of the results in \cite{Wang} and \cite{Ding} can be found in \cite{Alves2}  for $p$-Laplacian problem and in \cite{MZ} the fractional Laplacian problem.

In this paper, we focus on problem \eqref{p0} in the  case $k>2$, multi-critical terms then involved in. In the bounded domain $\Omega$,
the existence of positive solution of the multi-critical Sobolev-Hardy problem
 \begin{equation}\label{eq:1.3}
     \begin{aligned}
        \begin{cases}
         -\Delta u=\frac{\lambda}{|x|^s}u^{p-1}+\sum\limits_{i=1}^l\frac{\lambda_i}{|x|^{s_i}}u^{2^*(s_i)-1}+u^{2^*-1}\quad \text{in }\Omega,\\
         \quad\ \  u>0\quad \text{in }\Omega,\\
         \quad\ \  u=0\quad\text{on }\Omega
        \end{cases}
     \end{aligned}
 \end{equation}
 was studied in \cite{Z.Gao}, while the number of  positive solutions of the Choquard type problem
\begin{equation}
    \label{w4}
\begin{cases}
 -\Delta u=\mu |u|^{p-2}u+\sum\limits_{i=1}^{k}(|x|^{-(N-\alpha_i)}* |u|^{2^*_i})|u|^{2^*_i-2}u\quad \text{in}\  \Omega,\\
\quad\ \  u\in H_0^1(\Omega)
\end{cases}
\end{equation}
are described in \cite{LYY} by the Lusternik-Schnirelman category $cat_\Omega(\Omega)$ of the domain $\Omega$. Such a result seems difficult to establish for problem \eqref{eq:1.3}.

Inspired by \cite{Wang} and \cite{Ding}, we will show that problem \eqref{p0} has at least $cat_\Omega(\Omega)$ positive solutions based on the result in \cite{LYY}.

Our hypotheses  are  as follows.

$(V1)$ $V(x)\in \mathcal{C}(\mathbb{R}^N,\mathbb{R}), V(x)\geqslant0,$ and $\Omega:=\mbox{int}V^{-1}(0)$ is a nonempty bounded set with smooth boundary, and $\overline{\Omega}=V^{-1}(0)$.\\

$(V2)$ There exists $M_0>0$ such that
$$mes\{x\in \mathbb{R}^N:V(x)\leqslant M_0\}<\infty.$$

Let $\mu_1$ be the first eigenvalue of $-\Delta$ on $\Omega$ with zero Dirichlet condition.  We present our first result as follows.

\begin{theorem}
\label{th1}
  Suppose $(V1)$ and $(V2)$ hold, then for every $0<\mu<\mu_1$, there exists $\lambda(\mu)>0$ such that problem \eqref{p0} has at least one ground state positive solution for each $\lambda\geqslant\lambda(\mu)$.
\end{theorem}

Solutions of problem \eqref{p0} are found as critical points of the associated functional
\begin{equation}\label{I0}
\begin{split}
    I_{\lambda,\mu}(u)&=\frac{1}{2}\int_{\mathbb{R}^N}(|\nabla u|^2+\lambda V(x)u^2)\,dx-\frac{\mu}{p}\int_{\mathbb{R}^N}|u|^p\,dx\\
    &-\sum\limits_{i=1}^k\frac{1}{2  2^*_i}\int_{\mathbb{R}^N}\int_{\mathbb{R}^N}\frac{|u(x)|^{2^*_i}|u(y)|^{2^*_i}}{|x-y|^{N-\alpha_i}}\,dxdy\\
    \end{split}
\end{equation}
in the Hilbert space
\begin{equation}\label{eq:1.4}
E:=\{u\in H^1(\mathbb{R^N}):\int_{\mathbb{R}^N} V(x)u^2\,dx<\infty\}
\end{equation}
endowed with the norm
\begin{equation}\label{eq:1.5}
\|u\|^2=\|u\|^2_{H^1}+\int_{\mathbb{R}^N}V(x)u^2\,dx.
\end{equation}

Theorem \ref{th1} is proved by the mountain pass theorem. A crucial ingredient is to verify the $(PS)_c$ condition for the functional $I_{\lambda,\mu}(u)$. By a $(PS)_c$
condition we mean a sequence $\{u_n\}$ in $E$ satisfying $I_{\lambda,\mu}(u_n)\to c,\, I'_{\lambda,\mu}(u_n)\to 0$ as $n\to\infty$ contains a convergent subsequence.
We will show that $I_{\lambda,\mu}(u)$ satisfies $(PS)_c$ condition if $c< m(\mathbb{R}^N)$ with
\begin{equation}\label{m2}
    m(\mathbb{R}^N):=\inf\{J(u):u\in\mathcal{N}_{\mathbb{R}^N}\},
\end{equation}
where
\begin{equation}\label{I2}
    J(u)=\frac{1}{2}\int_{\mathbb{R}^N}|\nabla u|^2\,dx-\sum\limits_{i=1}^k\frac{1}{2  2^*_i}\int_{\mathbb{R}^N}\int_{\mathbb{R}^N}\frac{|u(x)|^{2^*_i}|u(y)|^{2^*_i}}{|x-y|^{N-\alpha_i}}\,dxdy
\end{equation}
and
\begin{equation}\label{N2}
    \mathcal{N}_{\mathbb{R}^N}:=\{u\in D^{1,2}(\mathbb{R}^N)\backslash\{0\}:\langle J'(u),u\rangle=0\}.
\end{equation}
It is shown in \cite{LYY} that $ m(\mathbb{R}^N)$ is uniquely achieved up to translations and dilations by the function
$$U(x)=C\left(\frac{\varepsilon}{\varepsilon^2+|x|^2}\right)^{\frac{N-2}{2}}$$
with $\varepsilon>0$, which is a solution for the limit problem
\begin{equation}
\label{p2}
 -\Delta u=\sum\limits_{i=1}^{k}(|x|^{-(N-\alpha_i)}* |u|^{2^*_i})|u|^{2^*_i-2}u\quad \text{in}\  \mathbb{R}^N.
\end{equation}
\bigskip

Next, we consider the limit behavior of solutions \eqref{p0}. Let $\lambda_n\to\infty$ and $\{u_n\}$ be the corresponding solutions of \eqref{p0}.
Then, we will show that $\{u_n\}$ concentrates at a solution of the limit problem
\begin{equation}\label{p1}
\begin{cases}
 -\Delta u=\mu |u|^{p-2}u+\sum\limits_{i=1}^{k}(|x|^{-(N-\alpha_i)}* |u|^{2^*_i})|u|^{2^*_i-2}u\quad \text{in}\  \Omega,\\
\quad\ \  u\in H_0^1(\Omega).
\end{cases}
\end{equation}

\begin{theorem}
\label{th2}
  Suppose (V1) and (V2) hold. Let $\{u_n\}$ be a sequence of solutions of \eqref{p0} such that $0<\mu<\mu_1,\lambda_n\to\infty$ and $I_{\lambda_n,\mu}(u_n)\to c<m(\mathbb{R}^N)$ as $n\to\infty$. Then, $\{u_n\}$ concentrates at a solution $u$ of \eqref{p1}, that is, $u_n\to u$ in $E$ as $n\to\infty$.
\end{theorem}

Finally, we derive from  the Lusternik-Schnirelman theory  a multiplicity result for problem \eqref{p0}.
\begin{theorem}
\label{th3}
  Suppose (V1) and (V2) hold and $N\geqslant4$. Then there exists $0<\mu^*<\mu_1$ and for each $0<\mu\leqslant\mu^*$, there is $\Lambda(\mu)>0$ such that problem \eqref{p0} has at least $cat_\Omega(\Omega)$ positive solutions whenever $\lambda \geqslant\Lambda(\mu)$.
\end{theorem}

The paper is organized as follows.  In section 2, we present preliminary results. Then we prove Theorems \ref{th1} and Theorem \ref{th2} in section 3. The multiple result in  Theorem \ref{th3} is established section 4.

\bigskip

\section{Preliminaries}

\bigskip

In the sequel, we denote by  $|u|_q$ the norm of $u$ in  $L^q(\mathbb{R}^N)$ and
$$\|u\|^2_{\lambda}=\|u\|^2_{H^1}+\lambda\int_{\mathbb{R}^N}V(x)u^2\,dx$$for $\lambda>0$, which is equivalent to that defined in \eqref{eq:1.5}.

Now, by the Hardy-Littlewood-Sobolev inequality in  \cite{Lieb}, we see that the functionals $I_{\lambda, \mu}(u)$ and $J(u)$ are well-defined and differentiable in $E$ and $D^{1,2}(\mathbb{R}^N)$ respectively.
\begin{lemma}\label{lem1}
  Let $\lambda_n\geqslant1$ and $u_n\in E$ be such that $\|u_n\|^2_{\lambda_n}\leqslant K$ for $\lambda_n\to\infty$. Then there is a $u\in H^1_0(\Omega)$ such that, up to a subsequence, $u_n\rightharpoonup u$ weakly in $E$ and $u_n\to u$ in $L^2(\mathbb{R}^N)$. Moreover, if $u=0$, we have $u_n\to 0 $ in $L^p(\mathbb{R}^N)$ for $2\leqslant p<\frac{2N}{N-2}$.
\end{lemma}

\begin{proof}
  Since for $\lambda_n\geqslant1$,  $\|u_n\|^2\leqslant\|u_n\|^2_{\lambda_n}\leqslant K$, we may assume $u_n\rightharpoonup u$ weakly in $E$ and $u_n\to u$ in $L^2_{loc}(\mathbb{R}^N)$. Set $C_m:=\{x: V(x)\geqslant\frac{1}{m}\}$ for $m\in \mathbb{N}$. Then
  $$\mathbb{R}^N\backslash\overline{\Omega}=\bigcup^{+\infty}_{m=1}C_m,$$
  and for all $m\in \mathbb{N}$,
  $$\int_{C_m}|u_n|^2\,dx\leqslant m\int_{C_m}V|u_n|^2\,dx =\frac{m}{\lambda_n}\int_{C_m}\lambda_nV|u_n|^2\,dx \leqslant\frac{m}{\lambda_n}\|u_n\|^2_{\lambda_n} {\color{red}\leqslant}\frac{m}{\lambda_n}K.$$
Hence $u_n\to 0$ in $L^2(C_m)$ as $n\to\infty$. By a diagonal process, we may assume that $u_n\to 0$ in $L^2(\mathbb{R}^N\backslash\overline{\Omega})$ as $n\to\infty$.

In order to show that $u_n\to u$ in $L^2(\mathbb{R}^N)$, we set for $R>0$ that
  $$A(R):=\{x\in\mathbb{R}^N:|x|>R,V(x)\geqslant M_0\}$$
  and
  $$B(R):=\{x\in\mathbb{R}^N:|x|>R,V(x)< M_0\},$$
  where $M_0$ is given in $(V2)$.
 Since  $A(R)\subset C_m$ if $m$ large, we have
  \begin{equation}\label{eq6}
      \int_{A(R)}|u_n|^2\,dx\leqslant\int_{C_m}|u_n|^2\,dx\to 0, \text{ as } n\to\infty.
  \end{equation}
  By the H\"older inequality,  for $1<p<\frac{N}{N-2},\frac 1p +\frac 1{p'} =1$ there holds
  \begin{align*}
    \int_{B(R)}|u_n-u|^2\,dx&\leqslant\left(\int_{\mathbb{R^N}}|u_n-u|^{2p}\,dx\right)^{\frac{1}{p}} |B(R)|^{\frac{1}{p'}}
    \leqslant C\|u_n-u\|^2_{\lambda_n}|B(R)|^{\frac{1}{p'}}=o_R(1),
  \end{align*}
  where $o_R(1)\to0$ as $R\to\infty$. Then, for $\varepsilon>0$, there exists $N>0$ such that for $n>N$,
  $$\int_{\mathbb{R}^N}|u_n-u|^2\,dx\leqslant\int_{B_R}|u_n-u|^2\,dx+\int_{A(R)}|u_n|^2\,dx+\int_{A(R)}|u|^2\,dx<\varepsilon.$$
  It yields that $u_n\to u$ in $L^2(\mathbb{R}^N)$.

Now, we turn to the last assertion. If $u_n\rightharpoonup 0$, for $2<p< \frac{2N}{N-2}$ and any subset $\tilde\Omega\subset \mathbb{R}^N$, the interpolation inequality yields
  \begin{equation}\label{eq8}
    \begin{aligned}
    \int_{\tilde \Omega}|u_n|^p\,dx&\leqslant C\left(\int_{\tilde \Omega}|\nabla u_n|^{2}\,dx\right)^{\frac{\theta p}{2}}\left(\int_{\tilde\Omega}|u_n|^2\,dx\right)^{\frac{(1-\theta)p}{2}}\\
    &\leqslant C\|u_n\|_{\lambda_n}^{\theta p}\left(\int_{A(R)}u_n^2\,dx+\int_{B(R)}u_n^2\,dx\right)^{\frac{(1-\theta)p}{2}}<\varepsilon,
  \end{aligned}
  \end{equation}
 for $n$ large, where $\theta=\frac{N(p-2)}{2p}$. Choosing $\tilde\Omega$ to be $A(R)$ and $B(R)$ respectively, we obtain $u_n\to 0$ in $L^p(\mathbb{R}^N)$ as $n\to\infty$.
\end{proof}

\bigskip

Let $L_{\lambda}:=-\Delta+\lambda V$ be the self-adjoint operator acting on $L^2(\mathbb{R})$ with the domain $E$ for all $\lambda\geqslant0$. We denote by $(\cdot,\cdot)$ the $L^2$-inner product and write
$$(L_{\lambda}u,v)=\int_{\mathbb{R}^N}(\nabla u\nabla v+\lambda V(x)uv)\,dx$$
for $u,v\in E$. It implies that
\begin{equation}\label{equ1}
    (L_{\lambda}u,u)=\int_{\mathbb{R}^N}(|\nabla u|^2+\lambda V(x)u^2)\,dx\leqslant \|u\|_{\lambda}^2.
\end{equation}
Denote by $a_{\lambda}:=\inf \sigma(L_{\lambda})$ the infimum of the spectrum of $L_{\lambda}$. It follows that
$$0\leqslant a_{\lambda}=\inf\{(L_{\lambda}u,u):u\in E,|u|_2=1\}$$
which is nondecreasing about the parameter $\lambda$. It is proved in \cite{Ding} the following result.
\begin{lemma}\label{lem2}
  For every $0<\mu<\mu_1$, there exists $\lambda(\mu)>0$ such that $$a_{\lambda}\geqslant\frac{\mu+\mu_1}{2}$$
  if  $\lambda\geqslant\lambda(\mu)$. Moreover,  for all $u\in E$ and $\lambda\geqslant\lambda(\mu)$ there exists a constant $C>0$ such that
  $$C\|u\|^2_{\lambda}\leqslant(L_{\lambda}u,u),$$
 where $C{\color{red}\leqslant}(\mu+\mu_1)/(2+\mu+\mu_1)$.
\end{lemma}

\bigskip

\bigskip

\section{Positive solutions and concentration}

\bigskip

\ \ In this section, we will prove the existence of positive solution for  problem \eqref{p0} and investigate  the limit behavior of the solution as the parameter $\lambda\to \infty$.

\begin{lemma}
\label{lem3}
  If $\mu\in(0,\mu_1)$ and $\lambda\geqslant\lambda(\mu)$, then every $(PS)_c$ sequence $\{u_n\}\subset E$ for $I_{\lambda,\mu}$ is bounded in $E$.
\end{lemma}
\begin{proof}
 Since  $2<p<2^*_{min}$, by Lemma \ref{lem2} we have
  \begin{equation*}
    \begin{aligned}
    c+1+\|u_n\|_{\lambda}&\geqslant I_{\lambda,\mu}(u_n)-\frac{1}{p}\langle I'_{\lambda,\mu}(u_n),u_n\rangle \\
    &\geqslant(\frac{1}{2}-\frac{1}{p})(L_\lambda u_n,u_n) +\sum\limits_{i=1}^k(\frac{1}{p}-\frac{1}{2\cdot2^*_i})\int_{\mathbb{R}^N}(|x|^{-(N-\alpha_i)}*|u_n|^{2^*_i})|u_n|^{2^*_i}\,dx\\
    &\geqslant(\frac{1}{2}-\frac{1}{p})C\|u_n\|_\lambda^2.
  \end{aligned}
  \end{equation*}
  This implies that $\{u_n\}$ is bounded in $E$.
\end{proof}

Now we verify that $ I_{\lambda,\mu}(u)$ has the mountain pass  geometry.

\begin{lemma}\label{lem5}
  For all $\mu\in(0,\mu_1)$ and $\lambda\geqslant\lambda(\mu)$, the functional $I_{\lambda,\mu}$ satisfies that

 \indent $(i)$\ $I_{\lambda,\mu}(0)=0$ and there exist $\rho,\delta>0$, such that $\inf\limits_{\|u\|=\rho}I_{\lambda,\mu}(u)\geqslant\delta>0$;

$(ii)$ there exists $e\in E$ such that $\|e\|_\lambda >\rho$ and $I_{\lambda,\mu}(e)\leqslant 0.$
\end{lemma}

\begin{proof} By the Hardy-Littlewood-Sobolev inequality, we have
$$
I_{\lambda,\mu}(u)\geqslant C(\|u\|_\lambda^2-\|u\|_\lambda^p-\sum\limits_{i=1}^k\|u\|_\lambda^{2\cdot2^*_i}).
$$
Then, $(i)$ follows by choosing $\|u\|_\lambda=\rho$ sufficiently small. $(ii)$ is valid since for any fixed $u\in E\backslash\{0\}$,
$I_{\lambda,\mu}(tu)\to -\infty$ as $t\to+\infty$.
\end{proof}

\bigskip

By Lemma \ref{lem5} and the mountain pass theorem, we know that there exists a $(PS)_c$ sequence of the functional $I_{\lambda,\mu}(u)$ at the mountain pass level
$$c_{\lambda,\mu}:=\inf\limits_{\gamma\in \Gamma_\lambda}\max_{t\in[0,1]}I_{\lambda,\mu}(\gamma(t)),$$
where
$$
\Gamma_\lambda:=\{\gamma\in\mathcal{C}([0,1],H^1(\mathbb{R}^N)):\gamma(0)=0,I_{\lambda,\mu}(\gamma(1))\leqslant0,\gamma(1)\neq0\}.
$$
Define
$$
m_{\lambda,\mu}:=\inf\limits_{u\in\mathcal{N}_{\lambda,\mu}}I_{\lambda,\mu}(u),
$$
and
$$
c_{\lambda,\mu}^s:=\inf\limits_{u\in H^1(\mathbb{R}^N)\backslash\{0\}}\sup_{t>0}I_{\lambda,\mu}(tu).
$$
We may show as \cite{willem} that
\begin{equation}\label{Equal}
m_{\lambda,\mu}=c_{\lambda,\mu}=c_{\lambda,\mu}^s.
\end{equation}

In the same spirit, we may define for the functional
\begin{equation}\label{I1}
    I_{\mu,\Omega}(u):=\frac{1}{2}\int_{\Omega}|\nabla u|^2\,dx-\frac{1}{p}\int_{\Omega}\mu|u|^p\,dx-\sum\limits_{i=1}^k\frac{1}{2  2^*_i}\int_{\Omega}\int_{\Omega}\frac{|u(x)|^{2^*_i}|u(y)|^{2^*_i}}{|x-y|^s}\,dxdy
\end{equation}
associated with problem \eqref{p1} the mountain pass level
$$
c_{\mu,\Omega}:=\inf\limits_{\gamma\in \Gamma_0}\max_{t\in[0,1]}I_{\mu,\Omega}(\gamma(t)),
$$
where
$$
\Gamma_0:=\{\gamma\in\mathcal{C}([0,1],H^1_0(\Omega)):\gamma(0)=0,I_{\mu,\Omega}(\gamma(1))\leqslant0,\gamma(1)\neq0\},
$$
and
$$
c_{\mu,\Omega}^s:=\inf\limits_{u\in H^1_0(\Omega)\backslash\{0\}}\sup_{t>0}I_{\mu,\Omega}(tu)
$$
as well as
\begin{equation}\label{m1}
    m_{\mu,\Omega}:=\inf\{I_{\mu,\Omega}(u):u\in\mathcal{N}_{\mu,\Omega}\},
\end{equation}
where
\begin{equation}\label{N1}
    \mathcal{N}_{\mu,\Omega}:=\{u\in H_0^1(\Omega)\backslash\{0\}:\langle I_{\mu,\Omega}'(u),u\rangle\}
\end{equation}
is corresponding Nehari manifold. Similarly,
$$m_{\mu,\Omega}=c_{\mu,\Omega}=c_{\mu,\Omega}^s.$$

\bigskip

\begin{lemma}\label{lem6}
  There exist $\tau=\tau(\mu)>0$ and a constant $\sigma(\mu)>0$ such that the mountain pass level $c_{\lambda,\mu}$ of $I_{\lambda,\mu}$ satisfies that
  $$\sigma(\mu)\leqslant c_{\lambda,\mu}<m(\mathbb{R}^N)-\tau$$
  for all $\lambda>0.$
\end{lemma}

\begin{proof} It is proved in \cite{LYY} that $c_{\mu,\Omega}<m(\mathbb{R}^N)$
and that there exists a critical point $u\in H_0^1(\Omega)$ of $I_{\mu,\Omega}$ such that  $I_{\mu,\Omega}(u)=c_{\mu,\Omega}$.  Extend the function $u$ to $\mathbb{R}^N$ such that $u=0$ outside $\Omega$,  then $u\in \mathcal{N}_{\lambda,\mu}$, which implies
$$c_{\lambda,\mu}\leqslant I_{\lambda,\mu}(u)=I_{\mu,\Omega}(u)=c_{\mu,\Omega}.$$
So there exists $\tau=\tau(\mu)>0$ such that $$c_{\lambda,\mu}\leqslant c_{\mu,\Omega}<m(\mathbb{R}^N)-\tau$$
for all $\lambda>0$.

Next, for each $u\in \mathcal{N}_{\lambda,\mu}$,
\begin{align*}
    0=\langle I'_{\lambda,\mu}(u),u \rangle&=(L_\lambda u,u)-\mu|u|^p_p-\sum\limits_{i=1}^k\int_{\mathbb{R}^N}(|x|^{-(N-\alpha_i)}*|u|^{2^*_i})|u|^{2^*_i}\,dx\\
    &\geqslant \|u\|_\lambda^2-C\|u\|_\lambda^p-C\sum\limits_{i=1}^k\|u\|_\lambda^{22^*_i},
\end{align*}
where $C$  depending only on $\mu$, yields that there exists $\sigma>0$ independent of $\lambda$ such that $\|u\|_\lambda\geqslant\sigma$,  and then
$$C\sigma\leqslant C\|u\|_\lambda^2\leqslant(L_\lambda u,u)=\mu|u|^p_p+\sum\limits_{i=1}^k\int_{\mathbb{R}^N}(|x|^{-(N-\alpha_i)}*|u|^{2^*_i})|u|^{2^*_i}\,dx.$$
Hence,
\begin{align*}
    c_{\lambda,\mu}&=m_{\lambda,\mu}=\inf\limits_{u\in \mathcal{N}_{\lambda,\mu}}I_{\lambda,\mu}(u)\\
    &\geqslant \inf\limits_{u\in \mathcal{N}_{\lambda,\mu}}(\frac{1}{2}-\frac{1}{p})\left[\mu|u|^p_p+\sum\limits_{i=1}^k\int_{\mathbb{R}^N}(|x|^{-(N-\alpha_i)}*|u|^{2^*_i})|u|^{2^*_i}\,dx\right]\\
    &\geqslant (\frac{1}{2}-\frac{1}{p})C\sigma,
\end{align*}
the conclusion follows.
\end{proof}

\bigskip

 Now we show  that $I_{\lambda,\mu}$ satisfies the $(PS)_c$ condition for certain $c$.
\begin{proposition}
\label{prop1}
  There exist $\mu\in (0,\mu_1)$ and $\lambda(\mu)>0$ such that $I_{\lambda,\mu}$ satisfies the $(PS)_c$ condition for $$c<m(\mathbb{R}^N)-\tau$$ whenever $\lambda\geqslant\lambda(\mu)$.
\end{proposition}

\begin{proof} It is known from Lemma \ref{lem3} that  the $(PS)_c$ sequence $\{u_n\}$  of  $I_{\lambda,\mu}$  is bounded in $E$.  We may assume that
  \begin{align*}
    &u_n\rightharpoonup u_0 \quad\text{in}\quad E,\\
    &u_n\to u_0\quad\text{in}\quad L^2_{loc}(\mathbb{R}^N),\\
    &u_n(x)\to u_0(x)\quad\text{a.e. on}\quad \mathbb{R}^N
  \end{align*}
as $n\to\infty$. By the hardy-Littlewood-Sobolev inequality  and H\"older inequality,
\begin{equation}\label{equ3}
    (|x|^{-(N-\alpha_i)}\ast |u_n|^{2^*_i})|u_n|^{2^*_i-1}\rightharpoonup(|x|^{-(N-\alpha_i)}\ast |u_0|^{2^*_i})|u_0|^{2^*_i-1}\quad \text{in}\ L^{\frac{2N}{N+\alpha_i}}(\mathbb{R}^N),
\end{equation}
 for $i=1,2,\cdots,k$. Thus, for every $\varphi\in H^1(\mathbb{R}^N)$,
  \begin{equation*}
     0=\lim_{n\to \infty}\langle I'_{\lambda,\mu}(u_n),\varphi\rangle =\langle I'_{\lambda,\mu}(u_0),\varphi\rangle,
  \end{equation*}
 that is, $u_0$ is a weak solution of problem \eqref{p0}.

Let $w_n=u_n-u_0$.  The Br\'ezis-Lieb type Lemma in \cite{Gao}
  \begin{equation}\label{eq10}
  \begin{aligned}
     &\lim\limits_{n\to\infty}\left[\int_{\mathbb{R}^N}(|x|^{-(N-\alpha_i)}\ast |u_n|^{2^*_i})|u_n|^{2^*_i-1}\,dx-\int_{\mathbb{R}^N}(|x|^{-(N-\alpha_i)}\ast |w_n|^{2^*_i})|w_n|^{2^*_i-1}\,dx\right]\\
     &=\int_{\mathbb{R}^N}(|x|^{-(N-\alpha_i)}\ast |u_0|^{2^*_i})|u_0|^{2^*_i-1}\,dx
  \end{aligned}
  \end{equation}
enables us to deduce
\[
 \begin{split}
  o(1)= \langle I'_{\lambda,\mu}(u_n),u_n\rangle &=\langle I'_{\lambda,\mu}(u_0),u_0\rangle+\langle I'_{\lambda,\mu}(w_n),w_n\rangle+o(1),\\
 \end{split}
 \]
 namely,
 \begin{equation}\label{eq11}
     \langle I'_{\lambda,\mu}(w_n),w_n\rangle=o(1).
 \end{equation}
 Similarly,
\begin{equation}
    \label{eq12}
   c+o(1) = I_{\lambda,\mu}(u_n)
   =I_{\lambda,\mu}(u_0)+I_{\lambda,\mu}(w_n)\geqslant I_{\lambda,\mu}(w_n).
 \end{equation}
since $I_{\lambda,\mu}(u_0)\geqslant0.$

In order to prove that $u_n$ converges strongly to $u_0$ in $E$, we claim that for each $u\in E$ the function $f_u(t)=I_{\lambda,\mu}(tu)$  has a unique critical point $t_u$ such that
 $ f_u(t_u) =\max_{t{\color{red}\geqslant}0}f(t)$ and $t_u u\in \mathcal{N}_{\lambda,\mu}$. Indeed,
 \begin{align*}
  f_u'(t)&=t(L_\lambda u,u)-t^{p-1}\mu|u|^p_p-\sum\limits_{i=1}^kt^{22^*_i-1}\int_{\mathbb{R}^N}(|x|^{-(N-\alpha_i)}*|u|^{2^*_i})|u|^{2^*_i}\,dx\\
  &=t\left[(L_\lambda u,u)-m_u(t)\right],
\end{align*}
where the function $m_u(t):=t^{p-2}\mu|u|^p_p+\sum\limits_{i=1}^kt^{22^*_i-2}\int_{\mathbb{R}^N}(|x|^{-(N-\alpha_i)}*|u|^{2^*_i})|u|^{2^*_i}\,dx>0$ satisfies  $\lim\limits_{t\to+\infty}m_u(t)=+\infty$ and $m_u(0)=0$. The claim follows readily.

Now we show $u_n$ converge strongly to $u_0$ in $E$.  Suppose, by contradiction, that $u_n$ does not converge strongly to $u_0$, then $w_n\not\to 0$ in $E$ as $n\to\infty$. So
there exists a unique $t_n>0$ such that $t_nw_n\in\mathcal{N}_{\lambda,\mu}$ and $f_{w_n}(t)$ achieves the maximum at $t=t_n$. By \eqref{eq11} and
\begin{equation}\label{eq13}
    \int_{\mathbb{R}^N}(|\nabla w_n|^2+\lambda V(x)w_n^2)\,dx=t_n^{p-2}\mu|w_n|^p_p+\sum\limits_{i=1}^kt_n^{22^*_i-2}\int_{\mathbb{R}^N}(|x|^{-(N-\alpha_i)}*|w_n|^{2^*_i})|w_n|^{2^*_i}\,dx{\color{red},}
\end{equation}
we find
 $$(t_n^{p-2}-1)\mu|w_n|^p_p+\sum\limits_{i=1}^k(t_n^{22^*_i-2}-1)\int_{\mathbb{R}^N}(|x|^{-(N-\alpha_i)}*|w_n|^{2^*_i})|w_n|^{2^*_i}\,dx=o(1),$$
which implies $t_n\to1$ since $w_n\not\to 0$ in $E$. Therefore,
\begin{equation}
    \label{eq14}
   I_{\lambda,\mu}(w_n)=I_{\lambda,\mu}(t_nw_n)+o(1).
\end{equation}

For any $\tilde\Omega\subset \mathbb{R}^N$, we define
\begin{equation}\label{J}
J_{\tilde\Omega}(u)=\frac{1}{2}\int_{\tilde\Omega}|\nabla u|^2\, dx-\sum\limits_{i=1}^k\frac{1}{22^*_i}\int_{\tilde\Omega}(|x|^{-(N-\alpha_i)}*|u|^{2^*_i})|u|^{2^*_i}\,dx
\end{equation}
and denote $J(u)= J_{\mathbb{R}^N}(u)$. Set $v_n=t_nw_n$. We may show similarly that the function $g_u(s)= J(su)$ for each $u\in E\setminus\{0\}$ has unique critical point $s_u>0$ which is the maximum point of $g_u(s)$. Particularly, there exists a unique $s_n\in\mathbb{R}^+$  such that $g_{v_n}'(s_n)=0$ and $s_nv_n\in\mathcal{N}_{\mathbb{R}^N}$. Since $v_n\in\mathcal{N}_{\lambda,\mu}$, we get
$$I_{\lambda,\mu}(v_n)=\sup_{s\geqslant0}I_{\lambda,\mu}(sv_n).$$
By Lemma \ref{lem1}, \eqref{eq12} and \eqref{eq14},
\begin{align*}
    m(\mathbb{R}^N)-\tau&\geqslant I_{\lambda,\mu}(u_n)\geqslant I_{\lambda,\mu}(w_n)=I_{\lambda,\mu}(t_nw_n)+o(1)\geqslant I_{\lambda,\mu}(s_nv_n)+o_n(1)\\
    &\geqslant J(s_nv_n)+o_n(1)+o_\lambda(1)
    \geqslant m(\mathbb{R}^N)+o_n(1)+o_\lambda(1).
\end{align*}
This yields a contradiction by choosing $\lambda$ large enough. Hence, $I_{\lambda,\mu}$ satisfies the $(PS)_c$ condition.
\end{proof}

\bigskip

\begin{proposition}
\label{prop2}
  If $u\in\mathcal{N}_{\lambda,\mu}$ is a critical point of $I_{\lambda,\mu}$ such that $I_{\lambda,\mu}(u)<2c_{\lambda,\mu}$, then $u$ does not change sign. Hence, $|u|$ is a solution of problem \eqref{p0}
\end{proposition}

\begin{proof}
  Let $u^{\pm}=\pm\max\{\pm u,0\}$, then $u=u^++u^-$. If neither $u^+$ nor $u^-$ is zero, we have $u^{\pm}\in\mathcal{M}_{\lambda,\mu}$ and
  $$I_{\lambda,\mu}(u)=I_{\lambda,\mu}(u^+)+I_{\lambda,\mu}(u^-)\geqslant2c_{\lambda,\mu},$$
  which is absurd. The proof is completed by the maximum principle.
\end{proof}

\bigskip

 Now, we are ready to prove Theorem \ref{th1}.

\emph{\bf Proof of Theorem \ref{th1}. }
   Let $\{u_n\}$ be a minimizing sequence of $I_{\lambda,\mu}$ restricting on $\mathcal{N}_{\lambda,\mu}$. By the Ekeland variational principle, we may assume that $\{u_n\}$ is a $(PS)_{c_{\lambda,\mu}}$ sequence, that is,
 $$I_{\lambda,\mu}(u_n)\to c_{\lambda,\mu}\quad\text{  and }\quad I'_{\lambda,\mu}(u_n)\to0.$$
 It follows from Proposition \ref{prop1} and Lemma \ref{lem6} that $\{u_n\}$  has a subsequence converging to a least energy solution $u_\lambda$ of \eqref{p0}. $\Box$

\bigskip

By Theorem \ref{th1}, for all $\lambda_n\geqslant\lambda(\mu)$, The problem
\begin{equation}\label{eq15}
    -\Delta u+\lambda_n V(x)u=\mu |u|^{p-2}u+\sum\limits_{i=1}^{k}(|x|^{-(N-\alpha_i)}* |u|^{2^*_i})|u|^{2^*_i-2}u\quad \text{in}\  \mathbb{R}^N
\end{equation}
has at least one ground state solution $u_{\lambda_n}$.
\bigskip

\emph{\bf Proof of Theorem \ref{th2}. }
Let $u_{\lambda_n}$ be a ground state solution to \eqref{eq15} for every $\lambda_n\geqslant\lambda(\mu)$.  Then we have $I_{\lambda_n,\mu}(u_{\lambda_n})=c_{\lambda_n,\mu}$ and $I'_{\lambda_n,\mu}(u_{\lambda_n})=0$. By Lemma \ref{lem6}, $c_{\lambda_n,\mu}$ is uniformly bounded.  We may show as Lemma \ref{lem3} that there is $C>0$ such that $\|u_{\lambda_n}\|^2_{\lambda_n}\leqslant C$. By Lemma \ref{lem1}, there exists $u\in H^1_0(\Omega)$ such that, up to a subsequence, $u_{\lambda_n}\rightharpoonup u$ weakly in $E$, $u_{\lambda_n}\to u$ in $L^2(\mathbb{R}^N)$ and  $u_{\lambda_n}\to u$ in $L^p(\mathbb{R}^N)$.

For each $\varphi\in H^1_0(\Omega)$, we extend $\varphi$ to $\mathbb{R}^N$ by setting $\varphi = 0$ outside $\Omega$. Then we have
\begin{equation}\label{eq16}
    0=\lim_{n\to\infty}\langle I'_{\lambda_n,\mu}(u_{\lambda_n}),\varphi\rangle=\langle I'_{\mu,\Omega}(u),\varphi\rangle.
    \end{equation}
This means that $u$ is a weak solution of problem \eqref{p1}.

Now we show $u_{\lambda_n}$ converges to $u$ in $E$. Suppose  on the contrary that $v_n:=u_{\lambda_n}-u{\color{red}\not\to}0$ in $E$ as $n\to\infty$.
By \eqref{eq10} and \eqref{eq16},
\[
   \lim_{n\to\infty} \langle I'_{\lambda_n,\mu}(v_n),v_n\rangle
    =\lim_{n\to\infty} \langle I'_{\lambda_n,\mu}(u_{\lambda_n}),u_{\lambda_n}\rangle-\langle I'_{\mu,\Omega}(u),u\rangle=0.
\]
Arguing as the proof of Proposition \ref{prop1}, there is unique $t_n > 0$ such that $t_nv_n\in\mathcal{N}_{\lambda_n,\mu}$ and $t_n\to1$ as $n\to\infty$. Hence,
$$I_{\lambda_n,\mu}(v_n)+o(1) =I_{\lambda_n,\mu}(t_nv_n) =\sup_{t\geqslant0}I_{\lambda_n,\mu}(tv_n).$$
On the other hand, there exists $s_n>0$ such that $s_nt_nv_n\in\mathcal{N}_{\mathbb{R}^N}$. We derive from Proposition \ref{prop1} that
\begin{align*}
m(\mathbb{R}^N)-\tau&>   c_{\lambda_n,\mu}+o(1)=I_{\lambda_n,\mu}(u_{\lambda_n})\geqslant I_{\lambda_n,\mu}(v_n)= I_{\lambda_n,\mu}(t_nv_n)+o(1)\\
    &\geqslant I_{\lambda_n,\mu}(s_nt_nv_n)+o(1)\geqslant J(s_nt_nv_n)+o(1)\geqslant m(\mathbb{R}^N)+o(1)
\end{align*}
a contradiction. The assertion follows. $\Box$

\bigskip

\begin{corollary}\label{coro2}
  If $\mu\in (0,\mu_1)$, then $\lim\limits_{\lambda_n\to\infty}c_{\lambda_n,\mu}=c_{\mu,\Omega}$.
\end{corollary}

\begin{proof} Obviously, $c_{\lambda_n,\mu}\leqslant c_{\mu,\Omega}.$ On the other hand,
by Theorems \ref{th1} and \ref{th2}, there exists $\{u_n\}$ such that $c_{\lambda_n,\mu} = I_{\lambda_n,\mu}(u_n)$
and  $\lim\limits_{\lambda_n\to\infty}I_{\lambda_n,\mu}(u_n)=I_{\mu,\Omega}(u)\geqslant m_{\mu,\Omega} = c_{\mu,\Omega}$.
\end{proof}

\bigskip

\section{Multiple solutions}

\bigskip

\ \ In this section, we show that problem \eqref{p0} has multiple solutions in connection with the domain $V^{-1}(0)=\Omega$ by Lusternik-Schnirelman theory. Such a problem is actually related to the limit problem \eqref{p1}. We know from \cite{LYY} that \eqref{p1} possesses at least $cat_\Omega(\Omega)$ solutions, and eventually so does problem \eqref{p0} as we will show.

Let us fix $r>0$ small enough so that
$$\Omega^+_{2r}:=\{x\in\mathbb{R}^N: dist(x,\Omega)<2r\}$$
and
$$\Omega^-_{r}:=\{x\in\mathbb{R}^N: dist(x,\partial\Omega)>r\}$$
are homotopically equivalent to $\Omega$. Denote $B_r(0):=\{x\in\mathbb{R}^N:|x|<r\}\subset\Omega$ and define $c_{\mu,r}:=c_{\mu,B_r(0)}$. Therefore, we may verify as  Lemma \ref{lem6} that for $\mu\in(0,\mu_1)$,
\begin{equation*}
  c_{\mu,\Omega}<c_{\mu,r}<m(\mathbb{R}^N)-\tau.
\end{equation*}

Let $\eta\in\mathcal{C}_c^\infty(\mathbb{R}^N)$ be such that $\eta(x) = x$ for all $x\in\Omega$. We introduce the barycenter of a function $0\neq u\in H^1(\mathbb{R}^N)$ as
$$\beta(u):=\sum\limits_{i=1}^k\int_{\mathbb{R}^N}\int_{\mathbb{R}^N}\frac{\eta(x)|u(x)|^{2^*_i}|u(y)|^{2^*_i}}{|x-y|^{N-\alpha_i}}\,dxdy
\Big/\int_{\mathbb{R}^N}\int_{\mathbb{R}^N}\frac{|u(x)|^{2^*_i}|u(y)|^{2^*_i}}{|x-y|^{N-\alpha_i}}\,dxdy.$$
Denote $I_{\mu,\Omega}^{c}=\{u\in \mathcal{N}_{\mu,\Omega}: I_{\mu,\Omega}\leqslant c\}$ the level set. It is proved in \cite{LYY} the following result.
\begin{lemma}\label{lem7}
  There exists $\mu^*\in(0,\mu_1)$ such that if $0<\mu<\mu^*$, then
  $$\beta(u)\in\Omega^+_r\quad\text{for every }\ u\in \cap I_{\mu,\Omega}^{c_{\mu,r}}.$$
\end{lemma}

\bigskip
For any domain $\tilde\Omega\subset \mathbb{R}^N$, we define
$$m(\tilde\Omega):=\inf\limits_{u\in\mathcal{N}_{\tilde\Omega}} J_{\tilde\Omega}(u),$$
where
 $$\mathcal{N}_{\tilde\Omega}:=\left\{u\in D_0^{1,2}(\tilde\Omega)\backslash\{0\}:\langle J_{\tilde\Omega}'(u), u\rangle=0\right\}.$$
\begin{lemma}\label{lem8a} There holds
\[
 m(\tilde\Omega)= m(\mathbb{R}^N),
 \]
and $m(\tilde\Omega)$ is never achieved  unless $\tilde\Omega= \mathbb{R}^N$.
\end{lemma}

\begin{proof} Obviously $m(\tilde\Omega)\geqslant m(\mathbb{R}^N)$. To prove the inverse inequality, let  $\{u_n\}\subset C^\infty_0(\mathbb{R}^N)$ be a minimizing
sequence of $m(\mathbb{R}^N)$. Choosing $y_n\in \mathbb{R}^N$ and $\lambda_n>0$ so that $v_n(x) = \lambda_n^{\frac{N-2}2}u_n(\lambda_n x+ y_n)\in C^\infty_0(\tilde\Omega)$, hence we
obtain $m(\tilde\Omega)\leqslant m(\mathbb{R}^N)$.

\end{proof}
\bigskip

\begin{lemma}\label{lem8}
  For $\mu\in(0,\mu_1)$, we have $\lim\limits_{\mu\to0}c_{\mu,r}=\lim\limits_{\mu\to0}c_{\mu,\Omega}.$
\end{lemma}

\begin{proof} For any  $\mu_n\in(0,\mu_1)$ and $\mu_n\to 0$ as $n\to\infty$, by \cite{LYY} there exists a solution $u_n$ of \eqref{p1} such that $I_{\mu_n,\Omega}(u_n)=c_{\mu_n,\Omega} $.
We may show as Lemma \ref{lem3} that $\{u_n\}$ is uniformly bounded in  $H^1_0(\Omega)$. Moreover, there is a unique $t_n\in \mathbb{R}^+$ such that $t_nu_n\in\mathcal{N}_{\Omega}$, that is,
\[
\int_\Omega|\nabla u_n|^2\,dx =\sum\limits_{i=1}^kt_n^{22^*_i-2}\int_{\Omega}(|x|^{-(N-\alpha_i)}*|u_n|^{2^*_i})|u_n|^{2^*_i}\,dx.
\]
Because $u_n$ is bounded, so does $t_n$.
Since $u_n\in\mathcal{N}_{\mu_n,\Omega}$, we deduce
\[
\begin{split}
&\sum\limits_{i=1}^kt_n^{22^*_i-2}\int_{\Omega}(|x|^{-(N-\alpha_i)}*|u_n|^{2^*_i})|u_n|^{2^*_i}\,dx\\
& =\mu_n\int_\Omega|u_n|^p\,dx+\sum\limits_{i=1}^k \int_{\Omega}(|x|^{-(N-\alpha_i)}*|u_n|^{2^*_i})|u_n|^{2^*_i}\,dx\\
&\geqslant\sum\limits_{i=1}^k \int_{\Omega}(|x|^{-(N-\alpha_i)}*|u_n|^{2^*_i})|u_n|^{2^*_i}\,dx
\end{split}
\]
implying $t_n\geqslant 1$.

Note that for $q>2$, the function $h(t)=\frac 12t^2-\frac 1qt^q$ is decreasing if $t\geqslant 1$ and increasing if $t\leqslant 1$. In particular, $h(t)$ is decreasing  for both $q=p$ and $q=22^*_i,i=1,2,\cdots k$ whenever  $t\geqslant1$. Therefore,
\[
\begin{split}
&I_{\mu_n,\Omega}(u_n)-I_{\mu_n,\Omega}(t_nu_n)\\
&=\left[\left(\frac{1}{2}-\frac{1}{p}\right)-\left(\frac{t_n^2}{2}-\frac{t_n^p}{p}\right)\right]\mu_n\int_\Omega|u_n|^p\,dx\\
&+\sum\limits_{i=1}^k\left[\left(\frac{1}{2}-\frac{1}{{22^*_i}}\right)-\left(\frac{t_n^2}{2}-\frac{t_n^{22^*_i}}{{22^*_i}}\right)\right]\int_{\Omega}\int_\Omega\frac{|u_n(x)|^{2^*_i}|u_n(y)|^{2^*_i}}{|x-y|^{N-\alpha_i}}\,dxdy\\
  \end{split}
  \]
 yields
 $$m(\Omega)\leqslant I_{\mu_n,\Omega}(t_nu_n)\leqslant I_{\mu_n,\Omega}(u_n)\leqslant c_{\mu_n,\Omega}+\frac{\mu_n}{p}\int_\Omega|t_nu_n|^p\,dx,$$
  and then
  \[
  m(\Omega)\leqslant \liminf_{n\to\infty}c_{\mu_n,\Omega}.
  \]

 On the other hand, for each $u\in \mathcal{N}_{\Omega}$, there exists $s_n>0$ such that $s_n u\in \mathcal{N}_{\mu_n,\Omega}$, alternatively,
 \[
 \begin{split}
&\sum\limits_{i=1}^k\int_{\Omega}(|x|^{-(N-\alpha_i)}*|u|^{2^*_i})|u|^{2^*_i}\,dx\\
& =\mu_ns_n^{p-2}\int_\Omega|u|^p\,dx+\sum\limits_{i=1}^k s_n^{22^*_i-2}\int_{\Omega}(|x|^{-(N-\alpha_i)}*|u|^{2^*_i})|u|^{2^*_i}\,dx\\
&\geqslant\sum\limits_{i=1}^k s_n^{22^*_i-2}\int_{\Omega}(|x|^{-(N-\alpha_i)}*|u|^{2^*_i})|u|^{2^*_i}\,dx
\end{split}
 \]
yielding $s_n\leqslant 1$. As a result,
\[
\begin{split}
c_{\mu_n,\Omega}= m_{\mu_n,\Omega}\leqslant I_{\mu_n,\Omega}(s_n u)\leqslant I_{\mu_n,\Omega}(u)\leqslant J(u),
\end{split}
\]
 which yields
 \[
 \limsup_{n\to\infty}c_{\mu_n,\Omega}\leqslant m(\Omega).
 \]
Consequently,  $m(\Omega)=\lim\limits_{n\to\infty}c_{\mu_n,\Omega}$. Similarly, $m(B_r)=\lim\limits_{n\to\infty}c_{\mu_n,B_r}$.
Lemma \ref{lem8a} then implies $$\lim\limits_{\mu\to0}c_{\mu,r}=\lim\limits_{\mu\to0}c_{\mu,\Omega}.$$
\end{proof}

\bigskip

Similarly, denote $I_{\lambda,\mu}^{c}=\{u\in \mathcal{N}_{\lambda,\mu}: I_{\lambda,\mu}\leqslant c\}$ the level set.
\begin{proposition}
\label{prop3}
 There exits $\mu^*\in(0,\mu_1)$ such that for each $0<\mu<\mu^*$ there is $\Lambda(\mu)>0$ so that if $u\in I_{\lambda,\mu}^{ c_{\mu,r}}$ there holds $\beta(u)\in\Omega^+_{2r}$ provided that $\lambda\geqslant \Lambda(\mu)$.
\end{proposition}

\begin{proof} We argue indirectly. Suppose on the contrary that for $\mu$ arbitrarily small,  there is a sequence $\{u_n\}$ such that $u_n\in \mathcal{N}_{\lambda_n,\mu}, \lambda_n\to\infty, I_{\lambda_n,\mu}(u_n)\to c\leqslant c_{\mu,r}$ and $\beta(u_n)\notin\Omega^+_{2r}$. We may show as Lemma \ref{lem3} that $\{u_n\}$ is uniformly bounded in $E$. By Lemma \ref{lem1}, there is $u_\mu\in H^1_0(\Omega)$ such that $u_n\rightharpoonup u_\mu$ weakly in $E$ and $u_n\to u_\mu$ in $L^2(\mathbb{R}^N)$.  The interpolation inequality yields $u_n\to u_\mu$ in $L^p(\mathbb{R}^N)$. Now, we distinguish two cases to discuss.

\emph{Case 1: }$\sum\limits_{i=1}^k\int_{\Omega}\int_{\Omega}\frac {|u_\mu(x)|^{2^*_i})|u_\mu(y)|^{2^*_i}}{|x-y|^{N-\alpha_i}}\,dx\leqslant\int_{\Omega}|\nabla u_\mu|^2\,dx-\mu\int_\Omega|u_\mu|^p\,dx.$

Let $w_n=u_n-u$.  First, we show $w_n\to 0$ in $E$. Indeed, were it not the case, we would have $w_n\not\to 0$ in $E$ as $n\to\infty$. By the Br\'{e}zis-Lieb type lemma,
  \begin{equation}
      \label{eq18}
      c+o(1)=I_{\lambda_n,\mu}(u_n)= I_{\lambda_n,\mu}(w_n) + I_{\lambda_n,\mu}(u_\mu)+o(1)\geqslant I_{\lambda_n,\mu}(w_n)+o(1)
  \end{equation}
  and  by the assumption,
\begin{equation} \label{eq19}
\begin{split}
 o(1) &= \langle I_{\lambda_n,\mu}'(u_n), u_n\rangle = \langle I_{\lambda_n,\mu}'(w_n), w_n\rangle\\ & +
 \int_{\Omega}|\nabla u_\mu|^2\,dx -\mu\int_{\Omega}|u_\mu|^p\,dx-\sum\limits_{i=1}^k\int_{\Omega}\int_{\Omega}\frac {|u_\mu(x)|^{2^*_i})|u_\mu(y)|^{2^*_i}}{|x-y|^{N-\alpha_i}}\,dx+o(1)\\
 &\leqslant\langle I_{\lambda_n,\mu}'(w_n), w_n\rangle + o(1).
 \end{split}
    \end{equation}
 That is,
  \begin{equation} \label{eq19}
  \begin{split}
 &\int_{\mathbb{R}^N}(|\nabla w_n|^2+\lambda V(x)w_n^2)\,dx\\
 &\leqslant\mu\int_{\mathbb{R}^N}|w_n|^p\,dx+\sum\limits_{i=1}^k\int_{\mathbb{R}^N}\int_{\mathbb{R}^N}\frac{|w_n(x)|^{2^*_i}|w_n(y)|^{2^*_i}}{|x-y|^{N-\alpha_i}}\,dxdy+o(1).\\
 \end{split}
    \end{equation}
As before, we can find  a unique constant $t_n>0$ such that  $t_nw_n\in\mathcal{N}_{\lambda_n,\mu}$, which and \eqref{eq19} yield
\begin{equation}
\label{eq20}
 \begin{aligned}
&t_n^{p-2}\mu|w_n|^p_p+\sum\limits_{i=1}^kt_n^{22^*_i-2}\int_{\mathbb{R}^N}\int_{\mathbb{R}^N}\frac{|w_n(x)|^{2^*_i}|w_n(y)|^{2^*_i}}{|x-y|^{N-\alpha_i}}\,dxdy\\
&\leqslant\mu|w_n|^p_p+\sum\limits_{i=1}^k\int_{\mathbb{R}^N}\int_{\mathbb{R}^N}\frac{|w_n(x)|^{2^*_i}|w_n(y)|^{2^*_i}}{|x-y|^{N-\alpha_i}}\,dxdy+o(1),
\end{aligned}
\end{equation}
then $t_n\leqslant1$. By \eqref{eq18} and \eqref{eq20},
\begin{equation*}
    \begin{aligned}
       &c+o(1)=I_{\lambda_n,\mu}(u_n)\\
       &\geqslant \left(\frac{1}{2}-\frac{1}{p}\right)\mu\int_\Omega|w_n|^p\,dx+\sum\limits_{i=1}^k\left(\frac{1}{2}-
       \frac{1}{{22^*_i}}\right)\int_{\Omega}\int_\Omega\frac{|w_n(x)|^{2^*_i}|w_n(y)|^{2^*_i}}{|x-y|^{N-\alpha_i}}\,dxdy+o(1)\\
       &\geqslant\left(\frac{1}{2}-\frac{1}{p}\right)\mu\int_\Omega|t_nw_n|^p\,dx+
       \sum\limits_{i=1}^k\left(\frac{1}{2}-\frac{1}{{22^*_i}}\right)t_n^{22^*_i}\int_{\Omega}\int_\Omega\frac{|w_n(x)|^{2^*_i}|w_n(y)|^{2^*_i}}{|x-y|^{N-\alpha_i}}\,dxdy+o(1)\\
       &=I_{\lambda_n,\mu}(t_nw_n)+o(1).
    \end{aligned}
\end{equation*}
Since $v_n:=t_nw_n\in\mathcal{N}_{\lambda_n,\mu}$, we have
\begin{equation}
\label{eq20a}
I_{\lambda_n,\mu}(v_n) = \sup_{s\geq0}I_{\lambda_n,\mu}(sv_n).
\end{equation}
Similarly, there exists a unique $s_n\in\mathbb{R}^+$ such that  $s_nv_n\in\mathcal{N}_{\mathbb{R}^N}$. It follows that
\begin{equation}
\label{eq20b}
 \begin{aligned}
&\sum\limits_{i=1}^ks_n^{22^*_i-2}\int_{\mathbb{R}^N}\int_{\mathbb{R}^N}\frac{|v_n(x)|^{2^*_i}|v_n(y)|^{2^*_i}}{|x-y|^{N-\alpha_i}}\,dxdy\\
&=\int_{\mathbb{R}^N}|\nabla v_n|^2\,dx\leqslant \int_{\mathbb{R}^N}(|\nabla v_n|^2+\lambda_nVv_n^2)\,dx\\
&=\mu\int_{\mathbb{R}^N}|v_n|^p\,dx+\sum\limits_{i=1}^k\int_{\mathbb{R}^N}\int_{\mathbb{R}^N}\frac{|v_n(x)|^{2^*_i}|v_n(y)|^{2^*_i}}{|x-y|^{N-\alpha_i}}\,dxdy,
\end{aligned}
\end{equation}
which implies $\{s_n\}$ is bounded. Therefore, by \eqref{eq20a} and the fact that $u_n\to u_\mu$ in $L^p(\mathbb{R}^N)$ we deduce
\begin{align*}
m(\mathbb{R}^N)-\tau> I_{\lambda_n,\mu}(v_n)+o_n(1)\geqslant I_{\lambda_n,\mu}(s_nv_n)+o_n(1)\geqslant J(s_nv_n) +  o_n(1)\geqslant m(\mathbb{R}^N)+o_n(1)
\end{align*}
a contradiction. Hence, $u_n\to u_\mu$ in $E$. As a result, $\beta(u_n)\to\beta(u_\mu)$. Moreover, by Lemma \ref{lem1}, $$I_{\mu,\Omega}(u_\mu)\leqslant\lim\limits_{n\to\infty}I_{\lambda_n,\mu}(u_n)\leqslant c_{\mu,r}.$$
Whence by Lemma \ref{lem7}, $\beta(u_\mu)\in\Omega^+_r$, which is a contradiction to $\beta(u_n)\notin\Omega^+_{2r}$.

   \emph{Case 2: } $\sum\limits_{i=1}^k\int_{\Omega}(|x|^{-(N-\alpha_i)}*|u_\mu|^{2^*_i})|u_\mu|^{2^*_i}\,dx>\int_{\Omega}|\nabla u_\mu|^2\,dx-\mu\int_\Omega|u_\mu|^p\,dx.$

  It is known that there exists $t_\mu>0$ such that $t_\mu u_\mu\in \mathcal{N}_{\mu,\Omega}$. By the assumption,
 \[
 \begin{split}
& t_\mu^{p-2}\mu\int_\Omega|u_\mu|^p\,dx + \sum\limits_{i=1}^kt_\mu^{22^*_i-2}\int_{\Omega}\int_{\Omega}\frac{|u_\mu(x)|^{2^*_i}|u_\mu(y)|^{2^*_i}}{|x-y|^{N-\alpha_i}}\,dxdy\\
&=\int_\Omega|\nabla u_\mu|^2\,dx < \mu\int_\Omega|u_\mu|^p\,dx + \sum\limits_{i=1}^k\int_{\Omega}\int_{\Omega}\frac{|u_\mu(x)|^{2^*_i}|u_\mu(y)|^{2^*_i}}{|x-y|^{N-\alpha_i}}\,dxdy,
 \end{split}
 \]
which  implies $t_\mu\in(0,1)$. Since  $p<2^*_{min}$, we have
\begin{align*}
       &c_{\mu,\Omega}\leqslant I_{\mu,\Omega}(t_\mu u_\mu)\\
       &\leqslant t_\mu^p\left[\left(\frac{1}{2}-\frac{1}{p}\right)\mu\int_\Omega|u_\mu|^p\,dx+\sum\limits_{i=1}^k\left(\frac{1}{2}
       -\frac{1}{{22^*_i}}\right)\sum\limits_{i=1}^k\int_{\Omega}\int_{\Omega}\frac{|u_\mu(x)|^{2^*_i}|u_\mu(y)|^{2^*_i}}{|x-y|^{N-\alpha_i}}\,dxdy\right]\\
       &\leqslant t_\mu^p\lim\limits_{n\to\infty}\left[\left(\frac{1}{2}-\frac{1}{p}\right)\mu\int_{\mathbb{R}^N}|u_n|^p\,dx+\sum\limits_{i=1}^k\left(\frac{1}{2}-\frac{1}{{22^*_i}}\right)
       \sum\limits_{i=1}^k\int_{\mathbb{R}^N}\int_{\mathbb{R}^N}\frac{|u_n(x)|^{2^*_i}|u_n(y)|^{2^*_i}}{|x-y|^{N-\alpha_i}}\,dxdy\right]\\
       &\leqslant\lim_{n\to\infty}I_{\lambda_n,\mu}(u_n)\leqslant c_{\mu,r}.\\
   \end{align*}
Thus, by Lemma \ref{lem8}, for $n$ large we obtain
\begin{equation}\label{eq21}
    \left|I_{\lambda_n,\mu}(u_n)-I_{\mu,\Omega}(t_\mu u_\mu)\right|\leqslant c_{\mu,r}-c_{\mu,\Omega}\to 0
\end{equation}
as $\mu\to 0$.
Observe that
\begin{align*}
       &I_{\lambda_n,\mu}(u_n)-I_{\mu,\Omega}(t_\mu u_\mu)\\
       =&\bigg(\frac{1}{2}-\frac{1}{p}\bigg)\mu\bigg(\int_{\mathbb{R}^N}|u_n|^p\,dx-t_\mu^p\int_\Omega|u_\mu|^p\,dx\bigg)\\
       +&\sum\limits_{i=1}^k\left(\frac{1}{2}-\frac{1}{{22^*_i}}\right)\bigg(\int_{\mathbb{R}^N}\int_{\mathbb{R}^N}\frac{|u_n(x)|^{2^*_i}|u_n(y)|^{2^*_i}}{|x-y|^{N-\alpha_i}}\,dxdy
       -t_\mu^{22^*_i}\int_{\Omega}\int_{\Omega}\frac{|u_n(x)|^{2^*_i}|u_n(y)|^{2^*_i}}{|x-y|^{N-\alpha_i}}\,dxdy\bigg),\\
   \end{align*}
   and both $|u_n|^p_p$ and $\int_\Omega|t_\mu u_\mu|^p\,dx$ are bounded, we infer from  \eqref{eq21} that
   \begin{align*}
     \int_{\mathbb{R}^N}\int_{\mathbb{R}^N}\frac{|u_n(x)|^{2^*_i}|u_n(y)|^{2^*_i}}{|x-y|^{N-\alpha_i}}\,dxdy
       -t_\mu^{22^*_i}\int_{\Omega}\int_{\Omega}\frac{|u_n(x)|^{2^*_i}|u_n(y)|^{2^*_i}}{|x-y|^{N-\alpha_i}}\,dxdy\to 0
   \end{align*}
   for $n$ large enough and as $\mu\to0,\,i=1,2,\cdots,k.$
   Extending $u_\mu$ to $\mathbb{R}^N$ by setting $u_\mu=0$ outside $\Omega$,  we see that
   \begin{align*}
   &\bigg|\int_{\mathbb{R}^N}\int_{\mathbb{R}^N}\eta(x)\left[\frac{|u_n(x)|^{2^*_i}|u_n(y)|^{2^*_i}}{|x-y|^{N-\alpha_i}}
   -\frac{|t_\mu u_\mu(x)|^{2^*_i}|t_\mu u_\mu(y)|^{2^*_i}}{|x-y|^{N-\alpha_i}}\right]\,dxdy\bigg|\\
   &\leqslant C\bigg|\int_{\mathbb{R}^N}\int_{\mathbb{R}^N}\left[\frac{|u_n(x)|^{2^*_i}|u_n(y)|^{2^*_i}}{|x-y|^{N-\alpha_i}}
   -\frac{|t_\mu u_\mu(x)|^{2^*_i}|t_\mu u_\mu(y)|^{2^*_i}}{|x-y|^{N-\alpha_i}}\right]\,dxdy\bigg|
   \end{align*}
  tends to $0$ as $\mu\to 0$. So for $\mu>0$ small enough and $n$ large, we have
   $$|\beta(u_n)-\beta(tu_\mu)|<r.$$
   It is known from  Lemma \ref{lem7} that $\beta(tu_\mu)\in\Omega^+_r$, whereas $\beta(u_n)\notin\Omega^+_{2r}$. The  contradiction completes the proof.
\end{proof}

\bigskip

 We know from  \cite{LYY} and \cite{Moroz2} that there exists a radially symmetric minimizer $v>0$ of $I_{\mu,B_r}$ on $\mathcal{N}_{\mu,B_r}$ for $\mu\in(0,\mu_1)$.
 It allows us to estimate the category of the level set of $I_{\lambda,\mu}$.
\begin{lemma}\label{cat}
  If $N\geqslant4$ and $\mu\in(0,\mu^*)$, for $\lambda\geqslant\Lambda(\mu)$, then
  $$cat_{I_{\lambda,\mu}^{ c_{\mu,r}}}(I_{\lambda,\mu}^{ c_{\mu,r}})\geqslant cat_\Omega(\Omega).$$
\end{lemma}

\begin{proof}
  Define $\gamma:\Omega^-_r\to I_{\lambda,\mu}^{ c_{\mu,r}}$ by
  \begin{align*}
    \gamma(y)(x)=\begin{cases}
      v(x-y),\quad&x\in B_r(y),\\
      0,\quad&x\notin B_r(y).
    \end{cases}
  \end{align*}
 We may verify that $\gamma(y)(x)\in\mathcal{N}_{\lambda,\mu}$, $I_{\lambda,\mu}(\gamma(y)(x))\leqslant c_{\mu,r}$ and $\beta\circ\gamma=id:\Omega^-_r\to\Omega^-_r$.

  Assume that $cat_{I_{\lambda,\mu}^{ c_{\mu,r}}}(I_{\lambda,\mu}^{ c_{\mu,r}})=n$, and
  $$I_{\lambda,\mu}^{ c_{\mu,r}}=\bigcup^{n}_{j=1}A_j,$$
  where $A_j,j=1,2,\cdots,n$, is closed and contractible in $I_{\lambda,\mu}^{ c_{\mu,r}}$, i.e. there exits $h_j\in\mathcal{C}([0,1]\times A_j,I_{\lambda,\mu}^{ c_{\mu,r}}$ such that, for every $u,v\in A_j$,
  $$h_j(0,u)=u,\quad h_j(1,u)=h_j(1,v).$$
 Let $B_j:=\gamma^{-1}(A_j),j=1,2,\cdots,n$.  For each $x\in \Omega^-_r$, $$\gamma(x)\in I_{\lambda,\mu}^{ c_{\mu,r}}\subset \bigcup^{n}_{j=1}A_j.$$
  So there exits $j_0$ such that $\gamma(x)\in A_{j_0}$, that is, $x\in\gamma^{-1}(A_{j_0})=B_{j_0}$. Therefore,
  $$\Omega^-_r\subset\bigcup^{n}_{j=1}B_j.$$

 For $x,y\in B_j, \gamma(x),\gamma(y)\in A_j$,
  the deformation
  $$g_j(t,x)=\beta_0(h_j(t,\gamma(x))), j=1,2,\cdots,n,$$
  fulfills
  $$g_j(0,x)=\beta_0(h_j(0,\gamma(x)))=\beta_0(\gamma(x))=x$$
  and
  $$g_j(1,x)=\beta_0(h_j(1,\gamma(x)))=\beta_0(h_j(1,\gamma(y)))=g_j(1,y).$$
 Hence,  $B_j$ is contractible in $\Omega^+_{2r}$. It follows that
  $$cat_\Omega(\Omega)=cat_{\Omega^+_{2r}}(\Omega^-_r)\leqslant \sum_{k=1}^{n}cat_{\Omega^+_{2r}}(B_k)=n.$$
\end{proof}

\bigskip
By Theorem 5.20 of \cite{willem}, we have the following lemma.
\begin{lemma}\label{n solutions}
  If $I_{\lambda,\mu}$ constraint to $\mathcal{N}_{\lambda,\mu}$ denoted by $I_{\lambda,\mu}\big|_{\mathcal{N}_{\lambda,\mu}}$ is bounded from below and satisfies the $(PS)_c$ condition for any $c\in[c_{\lambda,\mu},c_{\mu,r}]$, then $I_{\lambda,\mu}\big|_{\mathcal{N}_{\lambda,\mu}}$ has a minimum and $I_{\lambda,\mu}^{ c_{\mu,r}}$ contains at least $cat_{I_{\lambda,\mu}^{ c_{\mu,r}}}(I_{\lambda,\mu}^{ c_{\mu,r}})$ critical points of $I_{\lambda,\mu}\big|_{\mathcal{N}_{\lambda,\mu}}$.
\end{lemma}

\bigskip

\emph{\bf Proof of Theorem \ref{th3}.} For $0<\mu\leqslant\mu^*$ and $\lambda\geqslant\Lambda(\mu)$, we defined two maps
$$\Omega^-_r\stackrel{\gamma}{\longrightarrow}I_{\lambda,\mu}^{ c_{\mu,r}}\stackrel{\beta}{\longrightarrow}\Omega^+_{2r}.$$
The conclusion follows from  Proposition \ref{prop1},  Proposition \ref{prop3}, Lemma \ref{cat} and Lemma \ref{n solutions}. $\Box$

\section*{Acknowledgement}
This work was supported by  the National Natural Science Foundation of China (No: 12171212). 

\bibliographystyle{alpha}

\end{document}